\documentclass[graybox]{svmult}

\usepackage{amsfonts}
\usepackage{amsmath}
\usepackage{enumerate}
\usepackage{xcolor}
\usepackage{verbatim}

\newcommand{\R}{\mathbb{R}}
\newcommand{\C}{\mathbb{C}}
\newcommand{\N}{\mathbb{N}}
\newcommand{\EV}{\mathbb{E}}
\newcommand{\PP}{\mathbb{P}}
\newcommand{\Z}{\mathbb{Z}}

\newcommand{\ve}{\varepsilon}
\newcommand{\ch}{\mathbf 1}
\newcommand{\lan}{\langle}
\newcommand{\ran}{\rangle}
\DeclareMathOperator{\spa}{span}

\begin{document}

\title{Continuous frames and the Kadison-Singer problem}

\author{Marcin Bownik}

\institute{Institute of Mathematics, Polish Academy of Sciences, ul. Wita Stwosza 57,
80--952 Gda\'nsk, Poland and Department of Mathematics, University of Oregon, Eugene, OR 97403--1222, USA \email{mbownik@uoregon.edu}}

\maketitle



\abstract{
In this paper we survey a recent progress on continuous frames inspired by the solution of the Kadison-Singer problem \cite{KS} by Marcus, Spielman, and Srivastava \cite{MSS}. We present an extension of Lyapunov's theorem for discrete frames due to Akemann and Weaver \cite{AW} and a similar extension for continuous frames by the author \cite{Bow2}. 
We also outline a solution of the discretization problem, which was originally posed by Ali, Antoine, and Gazeau \cite{aag2}, and recently solved by Freeman and Speegle \cite{FS}. 
}


\section{From pure states to coherent states}

The solution of the Kadison-Singer problem by Marcus, Spielman, and Srivastava \cite{MSS} has had a great impact on several areas of analysis. This is due to the fact that the Kadison-Singer problem \cite{KS} was known to be equivalent to several well-known problems such as Anderson paving conjecture \cite{AA, An3}, Bourgain--Tzafriri restricted invertibility conjecture \cite{BT3}, Feichtinger's conjecture \cite{CCLV}, Weaver's conjecture \cite{We}. We refer to the survey  \cite{CT} and the papers \cite{Bow, BCMS, CT2} discussing the solution of the Kadison-Singer problem and its various ramifications.

The original formulation of the Kadison-Singer problem \cite{KS} asks whether a pure state on a maximal abelian self-adjoint algebra (MASA)  has a unique extension to the whole algebra of bounded operators $\mathcal B(\mathcal H)$ on a separable Hilbert space $\mathcal H$. In more concrete terms, let $\mathcal D \subset \mathcal B(\ell^2(\N))$ be the algebra of diagonal operators.  A state $s: \mathcal D \to \C$ is a positive  bounded linear functional $(A\ge 0 \implies s(A) \ge 0)$ such that $s(\mathbf I)=1$. A state is pure if it is not a convex combination of other states. The Kadison-Singer problem asks whether every pure state on $\mathcal D$ has a unique extension to a state on $\mathcal B(\ell^2(\N))$. 

In mathematical physics literature, there exists another meaning for a state, that is a coherent state. An authoritative treatment of coherent states and its various generalizations can be found in the book of Ali, Antoine, and Gazeau \cite{aag2}. Among several properties satisfied by canonical coherent states \cite[Chapter 1]{aag2}, they constitute an overcomplete family of vectors in the Hilbert space for the harmonic oscillator. In particular, coherent states satisfy an integral resolution of the identity, which naturally leads to the notion of a continuous frame. This is a generalization of the usual (discrete) frame, which was proposed independently by Ali, Antoine, and Gazeau \cite{aag} and by Kaiser \cite{Ka}, see also \cite{aag2, FR, GH}.

\begin{definition}\label{cf} Let $\mathcal H$ be a separable Hilbert spaces and let $(X,\mu)$ be a measure space. A family of vectors $\{\phi_t\}_{t\in X}$ is a {\it continuous frame} over $X$ for $\mathcal H$ if:
\begin{enumerate}[(i)]
\item for each $f\in \mathcal H$, the function $X \ni t \mapsto \langle f , \phi_t \rangle \in \C$ is measurable, and
\item there are constants $0<A \le B< \infty$, called {\it frame bounds}, such that 
\begin{equation}\label{cf1}
A||f||^2 \le \int_X |\langle f, \phi_t \rangle|^2 d\mu (t) \le B ||f||^2 \qquad\text{for all }f\in\mathcal H.
\end{equation}
\end{enumerate}
When $A=B$, the frame is called {\it tight}, and when $A=B=1$, it is a {\it continuous Parseval frame}. More generally, if only the upper bound holds in \eqref{cf1}, that is even if $A=0$, we say that $\{\phi_t\}_{t\in X}$ is a {\it continuous Bessel family} with bound $B$.
\end{definition}

Despite the fact that the notions of a pure state and a coherent state appear to be unrelated, the solution of Kadison-Singer problem has brought these two concepts much closer together. This is due to the discretization problem, which was proposed and popularized by Ali, Antoine, and Gazeau \cite[Chapter 17]{aag2}. Is it possible to obtain a discrete frame by sampling a continuous frame? Implicitly, some additional hypothesis is needed on continuous frame such us boundedness
\begin{equation}\label{cf2}
||\phi_t||^2 \le N \qquad \text{for all } t\in X.
\end{equation}

A partial answer to the discretization problem was given by Fornasier and Rauhut, see \cite[Remarks 4 and 5]{FR}. This was done by  constructing Banach spaces associated to continuous frames using the coorbit space theory developed by Feichtinger and Gr\"ochenig \cite{FG1, FG2}. In \cite[Theorem 5]{FR} they provide a general method to derive Banach frames and atomic decompositions for these Banach spaces by sampling the continuous frame. This yields the solution of the discretization problem for localized continuous frames satisfying certain integrability condition. 
 
A complete answer to the discretization problem was given by Freeman and Speegle \cite{FS}. Their method uses in an essential way the solution of Weaver's conjecture, which was shown in the landmark paper of  Marcus, Spielman, and Srivastava \cite{MSS}. In turn, Weaver \cite{We} has shown earlier that his conjecture is equivalent to the Kadison-Singer problem. Hence, the solution of the Kadison-Singer problem about pure states has paved the way for solving the discretization problem in the area of coherent states.

The solution of the discretization problem by Freeman and Speegle \cite{FS} relies on a sampling theorem for scalable frames. Scalable frames have been introduced by Kutyniok, Okoudjou, Philipp, and Tuley \cite{kopt}. A scalable frame $\{\phi_i\}_{i \in I}$ is a collection of vectors in $\mathcal H$ for which there exists a sequence of scalars $\{a_i\}_{i\in I}$ such that $\{a_i \phi_i \}_{i\in I}$ is a (Parseval) frame for $\mathcal H$. The concept of scalable frame is closely related to weighted frames. It is not hard to show that every continuous frame can be sampled to obtain a scalable frame. A much more difficult part is proving a sampling theorem for scalable frames. This result relies heavily on the solution of Weaver's conjecture \cite{We}.

In addition, we will also present Lyapunov's theorem for continuous frames which was recently shown by the author \cite{Bow2}. Every continuous frame defines a positive operator-valued measure (POVM) on $X$, see \cite{MHC}. To any measurable subset $E\subset X$, we assign a partial frame operator $S_{\phi,E}$ given by
\[
S_{\phi,E} f  = \int_E \lan f, \phi_t \ran \phi_t d\mu(t) \qquad\text{for } f\in\mathcal H.
\]
These are also known in the literature as localization operators, see e.g. \cite{Da1, Da2} for specific settings.
If the measure space $X$ is non-atomic, then the closure of the range of such POVM is convex. This is a variant of the classical Lyapunov's theorem which states that the range of a non-atomic vector-valued measure with values in $\R^n$ is a convex and compact subset of $\R^n$. 

Akemann and Weaver \cite{AW} have recently shown Lyapunov-type theorem for discrete frames. This result was also made possible by the solution of the Kadison-Singer problem. In fact, it can be considered as a significant strengthening of Weaver's conjecture \cite{We}. In contrast to Lyapunov-type theorem of Akemann and Weaver, Lyapunov's theorem for continuous frames on non-atomic measure spaces does not rely on the solution of the Kadison-Singer problem.

The paper is organized as follows. In Section \ref{S2} we present Lyapunov's theorem for continuous frames. In Section \ref{S3} we explain Lyapunov's theorem of Akemann and Weaver. In Section \ref{S4} we outline the proof of a sampling theorem for scalable frames which is then used in showing a sampling theorem for continuous frames. Finally, in Section \ref{S5} we present examples illustrating discretization of continuous frames.

\section{Lyapunov's theorem for continuous frames} \label{S2}

In this section we present the proof of Lyapunov's theorem for continuous frames due to the author \cite{Bow2}. We start with a preliminary result about continuous frames which is a consequence of the fact that we work with separable Hilbert spaces. The lower frame bound assumption is not essential and all of our results in this section hold for continuous Bessel families. 

\begin{proposition}\label{p1} Suppose that $\{\phi_t\}_{t\in X}$ is a continuous Bessel family in a separable Hilbert space $\mathcal H$. Then:
\begin{enumerate}[(i)]
\item
the support
$
\{t\in X: \phi_t \ne 0\}
$
is a $\sigma$-finite subset of $X$,
\item
 $\phi: X\to \mathcal H$ is a.e. uniform limit of a sequence of countably-valued measurable functions.
\end{enumerate}
\end{proposition}

\begin{proof}
Let $\{e_i\}_{i\in I}$ be an orthonormal basis of $\mathcal H$, where the index set $I$ is at most countable. For any $n\in \N$ and $i\in I$, by Chebyshev's inequality \eqref{cf1} yields
\[
\mu(\{t\in X: |\langle e_i, \phi_t \rangle|^2> 1/n \}) \le Bn<\infty.
\]
Hence, the set 
\[
\{t\in X: \phi_t \ne 0\}= \bigcup_{i\in I} \bigcup_{n\in\N} \{t\in X: |\langle e_i, \phi_t \rangle|^2> 1/n \} 
\]
is a countable union of sets of finite measure. This shows (i).

Since $\mathcal H$ is separable, by the Pettis Measurability Theorem \cite[Theorem II.2]{DU}, the weak measurability in Definition \ref{cf}(i) is equivalent to (Bochner) strong measurability on $\sigma$-finite measure spaces $X$. That is, $t \mapsto \phi_t$ is a pointwise a.e. limit of a sequence of simple measurable functions. Moreover, by \cite[Corollary II.3]{DU}, every measurable function $\phi: X\to \mathcal H$ is a.e. uniform limit of a sequence of countably-valued measurable functions. Although this result was stated in \cite{DU} for finite measure spaces, it also holds for $\sigma$-finite measure spaces. Since the support of $\{\phi_t\}_{t\in X}$ is $\sigma$-finite, we deduce (ii).
\end{proof}

It is convenient to define a concept of weighted frame operator as follows. This is a special case of a continuous frame multiplier introduced by Balazs, Bayer, and Rahimi \cite{BBR}; for a discrete analogue, see \cite{Ba}.

\begin{definition}
Suppose that $\{\phi_t\}_{t\in X}$ is a continuous Bessel family.
For any measurable function $\tau: X \to [0,1]$, 
define a {\it weighted frame operator}
\[
S_{\sqrt{\tau}\phi,X} f= \int_X \tau(t) \langle f, \phi_t \rangle \phi_t d\mu(t)
\qquad f\in \mathcal H.
\]
\end{definition}

Observe that
\[
\begin{aligned}
 \int_X |\langle f, \sqrt{\tau(t)} \phi_t \rangle|^2 d\mu (t) 
&  =  \int_X \tau (t) |\langle f,\phi_t \rangle|^2 d\mu (t) \\
& \le  \int_X |\langle f, \phi_t \rangle|^2 d\mu (t) 
 \le B ||f||^2.
\end{aligned}
\]
Hence, $\{\sqrt{\tau(t)} \phi_t\}_{t\in X}$ is a continuous Bessel family with the same bound as $\{\phi_t\}_{t\in X}$ and a weighted frame operator is merely the usual frame operator associated to $\{\sqrt{\tau(t)} \phi_t\}_{t\in X}$.
Using Proposition \ref{p1} we will deduce the following approximation result for continuous frames.

\begin{lemma}\label{approx}
Let $(X,\mu)$ be a measure space and let $\mathcal H$ be a separable Hilbert space.
Suppose that $\{\phi_t\}_{t\in X}$ is a continuous Bessel family in $\mathcal H$. Then for every $\ve>0$, there exists a continuous Bessel family $\{\psi_t\}_{t\in X}$, which takes only countably many values, such that:
\begin{enumerate}[(i)]
\item
there exists a partition $\{X_n\}_{n\in \N}$ of $X$ into measurable sets and a sequence $\{t_n\}_{n\in\N} \subset X$, such that $t_n \in X_n$ and
\begin{equation}\label{ap1}
\psi_t = \phi_{t_n} \qquad\text{for a.e. }t \in X_n, \ n\in\N,
\end{equation}
\item
for any measurable function $\tau: X \to [0,1]$ we have
\begin{equation}\label{ap2}
||S_{\sqrt{\tau}\phi,X} - S_{\sqrt{\tau}\psi,X}||<\ve.
\end{equation}
\end{enumerate}
\end{lemma}

\begin{proof}
By Proposition \ref{p1}(i) we can assume that $(X,\mu)$ is $\sigma$-finite and $\phi_t \ne 0$ for all $t\in X$. Then the measure space $X$ can be decomposed into its atomic $X_{at}$ and non-atomic $X \setminus X_{at}$ parts. Since $X$ is $\sigma$-finite, it has at most countably many atoms. Since every measurable mapping is constant a.e. on atoms, we can take $\psi_t = \phi_t$ for all $t\in X_{at}$, and the conclusions (i) and (ii) hold on $X_{at}$. Therefore, without loss of generality can assume that $\mu$ is a non-atomic measure.

Define measurable sets $Y_0=\{t\in X: ||\phi_t||<1\}$ and
\[
Y_n= \{t\in X: 2^{n-1} \le ||\phi_t||< 2^n \}, \qquad n\ge 1.
\]
Then, for any $\ve>0$, we can find a partition $\{Y_{n,m}\}_{m\in \N}$ of each $Y_n$ such that $\mu(Y_{n,m}) \le 1$ for all $m\in\N$. By Proposition \ref{p1}(ii) applied to each family $\{\phi_t \}_{t\in Y_{n,m}}$, we can find a countably-valued measurable function $\{ \tilde \psi_t\}_{t\in Y_{n,m}}$ such that 
\begin{equation}\label{ss1}
||\tilde \psi_t - \phi_t || \le \frac{\ve}{4^n 2^{m+1}} \qquad\text{for a.e. }t\in Y_{n,m}.
\end{equation}
Since $\{Y_{n,m}\}_{n\in \N_0, m\in \N}$ is a partition of $X$, we obtain a global countably-valued  function $\{\tilde \psi_t\}_{t\in X}$ satisfying \eqref{ss1}. Thus, we can partition $X$ into countable family of measurable sets $\{X_k\}_{k\in\N}$ such that $\{\tilde \psi_t\}_{t\in X}$ is constant on each $X_k$. Moreover,we can also require that $\{X_k\}_{k\in \N}$ is a refinement of a partition $\{Y_{n,m}\}_{n\in \N_0, m\in \N}$.

For a fixed $k\in \N$,  take $n$ and $m$ such that $X_k \subset Y_{n,m}$. Choose $t_k \in X_k$ for which \eqref{ss1} holds. Define a countably-valued function $\{\psi_t\}_{t\in X}$ by
\[
\psi_t = \phi_{t_k} \qquad \text{for }t \in X_k, \ k\in \N.
\]
Thus, the conclusion (i) follows by the construction.

Now fix $n\in \N_0$ and $m\in \N$, and take any $t\in Y_{n,m}$ outside the exceptional set in \eqref{ss1}. Let $k\in \N$ be such that $t\in X_k$. By \eqref{ss1},
\[
||\psi_t - \phi_t || = ||\phi_{t_k} - \phi_{t}|| \le ||\phi_{t_k} - \tilde \psi_{t_k}||+ ||\tilde \psi_t - \phi_t|| \le 2 \frac{\ve}{4^n 2^{m+1}}.
\]
Thus,
\begin{equation}\label{ss2}
||\psi_t - \phi_t || \le \frac{\ve}{4^n 2^m} \qquad\text{for a.e. }t\in Y_{n,m}.
\end{equation}

Take any $f\in\mathcal H$ with $||f||=1$.
Then, for a.e. $t\in Y_{n,m}$,
\[
\begin{aligned}
||\langle f, \psi_t \rangle|^2 & - |\langle f, \phi_t \rangle|^2|
 = |\langle f, \psi_t-\phi_t \rangle||\langle f, \psi_t +\phi_t \rangle| 
\\
&\le ||\psi_t - \phi_t||(||\psi_t||+||\phi_t||)
\le \frac{\ve}{4^n 2^m}(2^n+\ve + 2^n) \le \frac{3\ve}{2^n2^m} .
\end{aligned}
\]
Integrating over $Y_{n,m}$ and summing over $n\in\N_0$ and $m\in\N$ yields
\[
\int_X ||\langle f, \psi_t \rangle|^2-|\langle f, \phi_t \rangle|^2 | d\mu(t)  \le  \sum_{n=0}^\infty\sum_{m=1}^\infty \frac{3\ve}{2^n2^m} \mu(Y_{n,m}) \le 6\ve.
\]
Using the fact that $S_{\sqrt{\tau}\phi,X}$ is self-adjoint, we have
\[
\begin{aligned}
||S_{\sqrt{\tau}\phi,X} - S_{\sqrt{\tau}\psi,X}||
& = \sup_{||f||=1} |\langle (S_{\sqrt{\tau}\phi,X} - S_{\sqrt{\tau}\psi,X})f,f \rangle |
\\
&= \sup_{||f||=1} \bigg| \int_X \tau(t) (|\langle f, \psi_t \rangle|^2-|\langle f, \phi_t \rangle|^2) d\mu(t)  \bigg| \le 6\ve.
\end{aligned}
\]
Since $\ve>0$ is arbitrary, this completes the proof.
\end{proof}

\begin{remark}\label{rem}
Suppose $\{\psi_t\}_{t\in X}$ is a continuous frame which takes only countably many values as in Lemma \ref{approx}. Then for practical purposes, such a frame can be treated as a discrete frame. Indeed, there exists a partition $\{X_n\}_{n\in\N}$ of $X$  and a sequence $\{t_n\}_{n\in \N}$ such that \eqref{ap1} holds.
Since $\{\psi_t\}_{t\in X}$ is Bessel, we have $\mu(X_n)<\infty$ for all $n$ such that $\phi_{t_n} \ne 0$. Define vectors
\[
\tilde \phi_n =\sqrt{\mu(X_n)}\phi_{t_n} \qquad  n\in\N.
\]
Then, for all $f\in \mathcal H$,
\begin{equation}\label{rem3}
\int_X |\langle f, \psi_t \rangle|^2 d\mu (t) = \sum_{n\in \N} \int_{X_n} |\langle f, \phi_{t_n} \rangle|^2 d\mu (t) = \sum_{n\in \N} |\langle f, \tilde \phi_n \rangle|^2.
\end{equation}
Hence, $\{ \tilde  \phi_n \}_{n\in\N}$ is a discrete frame and its frame operator coincides with that of a continuous frame $\{\psi_t\}_{t\in X}$. This observation will be used in a subsequent theorem and also in Section \ref{S4}.
\end{remark}

\begin{theorem}\label{lya}
Let $(X,\mu)$ be a non-atomic  measure space.
Suppose that $\{\phi_t\}_{t\in X}$ is a continuous Bessel family in $\mathcal H$.  For any measurable function $\tau: X \to [0,1]$, 
consider a weighted frame operator
\[
S_{\sqrt{\tau}\phi,X} f= \int_X \tau(t) \langle f, \phi_t \rangle \phi_t d\mu(t)
\qquad f\in \mathcal H.
\]
Then, for any $\ve>0$, there exists a measurable set $E \subset X$ such that
\begin{equation}\label{lya1}
||S_{\phi,E} - S_{\sqrt{\tau}\phi,X}||<\ve.
\end{equation}
\end{theorem}

\begin{proof}
Let $\{\psi_t\}_{t\in X}$ be a continuous Bessel family as in Lemma \ref{approx}. Thus, there exists a partition $\{X_n\}_{n\in \N}$ of $X$ into measurable sets and a sequence $\{t_n\}_{n\in\N} \subset X$, such that $t_n \in X_n$ and \eqref{ap1} holds.
Since $\{\psi_t\}_{t\in X}$ is Bessel, we have $\mu(X_n)<\infty$ for all $n$ such that $\phi_{t_n} \ne 0$.
By Remark \ref{rem} the continuous frame $\{\psi_t\}_{t\in X}$ is equivalent to a discrete frame 
\[
\{\tilde \phi_n =\sqrt{\mu(X_n)}\phi_{t_n} \}_{n\in\N}.
\]
More precisely, for any measurable function $\tau: X \to [0,1]$, the frame operator $S_{\sqrt{\tau}\psi,X}$ of a continuous Bessel family $\{\sqrt{\tau(t)}\psi_t\}_{t\in X}$ coincides with the frame operator of a discrete Bessel sequence
\begin{equation}\label{lya4}
\{\sqrt{\tau_n} \phi_{t_n}\}_{n\in \N} \qquad\text{where }
 \tau_n=\int_{X_n}\tau(t) d\mu(t).
\end{equation}
Indeed, for all $f\in\mathcal H$,
\begin{equation}\label{lya6}
\begin{aligned}
\int_X |\langle f, \sqrt{\tau(t)} \psi_t \rangle|^2 d\mu (t)& = \sum_{n\in \N} \int_{X_n} \tau(t) |\langle f, \psi_t \rangle|^2 d\mu (t) \\
&=  \sum_{n\in \N} \tau_n  |\langle f,  \phi_{t_n} \rangle|^2 = \sum_{n\in \N} |\langle f, \sqrt{\tau_n} \phi_{t_n} \rangle|^2.
\end{aligned}
\end{equation}

Since $\mu$ is non-atomic, we can find subsets $E_n \subset X_n$ be such that $\mu(E_n)=\tau_n$. Define $E= \bigcup_{n\in \N} E_n$. 
Then, a simple calculation shows that
\begin{equation}\label{lya8}
S_{\psi,E}= S_{\sqrt{\tau}\psi,X}.
\end{equation}
Indeed, by \eqref{lya6}
\[
\langle S_{\sqrt{\tau}\psi,X}f,f \rangle = \sum_{n\in \N} \tau_n  |\langle f, \phi_{t_n} \rangle|^2 = \sum_{n\in \N} \int_{E_n} |\langle f, \psi_t \rangle|^2 d\mu (t)=
\langle S_{\psi,E} f,f \rangle.
\]
Hence, by \eqref{ap2} and \eqref{lya8}
\[
||S_{\phi,E} - S_{\sqrt{\tau}\phi,X}|| 
\le
||S_{\phi,E} - S_{\psi,E}||+||S_{\sqrt{\tau}\psi,X} - S_{\sqrt{\tau}\phi,X}||
 \le 2\ve.
\]
Since $\ve>0$ is arbitrary, this shows \eqref{lya1}. 
\end{proof}

Theorem \ref{lya} implies the Lyapunov theorem for continuous frames. Theorem \ref{lyu} is in a spirit of Uhl's theorem \cite{Uh}, which gives sufficient conditions for the convexity of the closure of the range of a non-atomic vector-valued measure, see also \cite[Theorem IX.10]{DU}. Note that the positive operator valued measure (POVM), which is given by $E \mapsto S_{\phi, E}$, does not have to be of bounded variation. Hence, Theorem \ref{lyu} can not be deduced from Uhl's theorem.

\begin{theorem}\label{lyu}
Let $(X,\mu)$ be a non-atomic measure space.
Suppose that $\{\phi_t\}_{t\in X}$ is a continuous Bessel family in $\mathcal H$. Let $\mathcal S$ be the set of all partial frame operators 
\begin{equation}\label{lyu1}
\mathcal S=\{ S_{\phi, E}: E \subset X \text{ is measurable} \}
\end{equation}
Then, the operator norm closure $\overline{\mathcal S} \subset \mathcal B(\mathcal H)$ is convex.
\end{theorem}

\begin{proof}
Note that set
\[
\mathcal T= \{S_{\sqrt{\tau}\phi,X}: \tau \text{ is any measurable }X \to [0,1]\}
\]
is a convex subset of $\mathcal B(\mathcal H)$. Hence, its operator norm closure $\overline{\mathcal T}$ is also convex. If $\tau=\ch_E$ is a characteristic function on $E \subset X$, then $S_{\sqrt{\tau}\phi,X}=S_{\phi, E}$. Hence, $\mathcal S \subset \mathcal T$. By Theorem \ref{lya} their closures are the same $\overline{\mathcal S}=\overline{\mathcal T}$.
\end{proof}

Theorem \ref{lyu} can be extended to POVMs given by measurable mappings with values in positive compact operators.

\begin{definition}\label{cov}
Let $\mathcal K_+(\mathcal H)$ be the space of positive compact operators on a separable Hilbert space $\mathcal H$. Let $(X, \mu)$ be a measure space. We say that $T = \{T_t\}_{t\in X}: X \to \mathcal K_+(\mathcal H)$ is {\it compact operator-valued Bessel family} if:
\begin{enumerate}[(i)]
\item 
for each $f,g \in\mathcal H$, the function $X \ni t \mapsto \lan T_t f, g \ran \in \C$ is measurable, and 
\item
there exists a constant $B>0$ such that
\[
\int_X \lan T_t f,f \ran \le B ||f||^2 \qquad \text{for all }f\in \mathcal H.
\]
\end{enumerate}
\end{definition}

For $\phi \in \mathcal H$, let $\phi\otimes\phi$ denote a rank one operator given by
\[
(\phi\otimes\phi)(f) = \langle f, \phi \rangle \phi 
\qquad\text{for }f\in\mathcal H.
\]
Observe that if $\{\phi_t\}_{t\in X}$ is a continuous Bessel family, then $T_t=\phi_t \otimes \phi_t$ is an example of compact operator-valued Bessel family. This corresponds to rank 1 operator-valued mappings. Since finite rank operators are a dense subset of $\mathcal K_+(\mathcal H)$ with respect to the operator nom, the space $\mathcal K_+(\mathcal H)$ is separable. It turns out that Theorem \ref{lyu} also holds in a more general setting. The proof is an adaption of the above arguments and can be found in \cite{Bow2}.

\begin{theorem}\label{cov8}
Suppose that $\{T_t\}_{t\in X}$ is a compact operator-valued Bessel family over a non-atomic measure space $(X,\mu)$. Define a positive operator-valued measure $\Phi$ on $X$ by
\begin{equation}\label{cov8a}
\Phi(E) = \int_E T_t d\mu(t) \qquad\text{for measurable } E \subset X.
\end{equation}
Then, the closure of the range of $\Phi$ is convex.
\end{theorem}

However, there is a definite limitation how far one can extend Lyapunov's theorem in this direction. For example, the assumption that the Bessel family $\{T_t\}_{t\in X}$ in Theorem \ref{cov8} is compact-valued is necessary, see \cite{Bow2}.

\section{Discrete frames and approximate Lyapunov's theorem} \label{S3}

Akemann and Weaver \cite{AW} have shown an interesting generalization of Weaver's $KS_r$ Conjecture \cite{We} in the form of approximate Lyapunov theorem. This was made possible thanks to the breakthrough solution of the Kadison-Singer problem \cite{CT, KS} by Marcus, Spielman, and Srivastava \cite{MSS}. 

Hence, if $\{\phi_i\}_{i\in I}$ in $\mathcal H$ is a frame (or more generally Bessel sequence), then its frame operator is given
\[
S = \sum_{i\in I} \phi_i \otimes \phi_i.
\]
In particular, if $\phi \in \mathcal H = \C^d$, then $\phi \otimes \phi$ is represented by $d\times d$ matrix $\phi \phi^*$, where $\phi$ is treated as a column vector and $\phi^*$ is its adjoint, a row vector.

The main result of \cite{MSS} takes the following form. The special case was shown by Casazza, Marcus, Speegle, and the author \cite{BCMS}.

\begin{theorem}\label{thmp}
Let $\epsilon > 0$. Suppose that $v_1, \dots, v_m$ are jointly independent random vectors in $\C^d$, which take finitely many values and satisfy  
\begin{equation}\label{thmp1}
\sum_{i=1}^m \EV[v_{i}v_{i}^*] = \mathbf I
\qquad\text{and}\qquad
\EV[ \|v_{i}\|^{2}] \le \epsilon \quad\text{for all } i.
\end{equation}
Then,
\begin{equation}\label{thmp2}
\PP \bigg(\bigg\| \sum_{i=1}^m v_i v_i^* \bigg\| \leq (1 + \sqrt{\epsilon})^2 \bigg) > 0.
\end{equation}
In the special case when  $v_1,\ldots,v_m$ take at most two values and $\epsilon<1/2$, we have
\[
\PP\bigg( \bigg\| \sum_{i=1}^m v_i v_i^* \bigg\| \leq 1 + 2 \sqrt{\epsilon} \sqrt{1-\epsilon}\bigg)  > 0 .
\]
\end{theorem}

Theorem \ref{thmp} implies Weaver's $KS_r$ conjecture. We state it in a form formulated by Akemann and Weaver \cite[Lemma 2.1]{AW}.

\begin{lemma}\label{MSS2}
Let $\{u_i\}_{i \in [m]}$ in  $\C^d$ be a Parseval frame
\begin{equation}\label{mss3}
\sum_{i=1}^m u_i u_i^* = \mathbf I \qquad\text{and}\qquad\|u_i\|^2 \le \delta
\quad\text{for all }i.
\end{equation}
Let $r\in\N$ and $t_1, \ldots, t_r>0$ satisfy $\sum_{k=1}^r t_k=1$. Then, there exists a partition $\{I_1,\ldots, I_r\}$ of $[m]$ such that each $\{u_i\}_{i\in  I_k}$, $k=1,\ldots,r$, is a Bessel sequence with the  bounds
\begin{equation}\label{mss5}
\bigg\|\sum_{i\in I_k} u_i u_i^* \bigg\| \le t_k (1 + \sqrt{r\delta})^2.
\end{equation}
\end{lemma}

\begin{proof} Assume $\{u_i\}_{i \in [m]}$ in $\C^d$ satisfies \eqref{mss3}. For any $r\in \N$, let $v_1,\ldots,v_m$ be independent random vectors in $(\C^d)^{\oplus r}=\C^{rd}$ such that each vector $v_i$ takes $r$ values
\[
(t_1)^{-1/2} \begin{bmatrix}  u_i \\ 0 \\ \vdots  \\0 \end{bmatrix}, \ldots,
(t_k)^{-1/2} \begin{bmatrix}    0 \\ \vdots  \\ 0 \\  u_i \end{bmatrix}
\]
with probabilities $t_1,\ldots,t_r$, respectively. Then,
\[
\sum_{i=1}^m \EV [v_i v_i^*] 
=  \begin{bmatrix} \sum_{i=1}^m u_i u_i^* &   &  \\   & \ddots &  \\  &  & \sum_{i=1}^m u_i u_i^*\end{bmatrix} = 
\begin{bmatrix} \mathbf I_d &   &  \\   & \ddots &  \\  &  & \mathbf I_d \end{bmatrix} =\mathbf I_{dr},
\]
and
\[
\EV[||v_i||^2 ] =  r||u_i||^2 \le \epsilon:=r\delta. 
\]
Hence, \eqref{thmp1} holds and Theorem \ref{thmp} yields \eqref{thmp2}. Choose an outcome for which the bound in \eqref{thmp2} happens. For this outcome define
\[
I_k = \{i \in [m]: v_i \text{ is  non-zero in $k^{\rm th}$ entry} \}, \qquad\text{for } k=1,\ldots,r.
\]
Thus, the block diagonal matrix
\[
\sum_{i=1}^m v_i v_i^* 
= \begin{bmatrix} \frac1{t_1} \sum_{i\in I_1} u_i u_i^* &   &  \\   & \ddots &  \\  &  & \frac1{t_r} \sum_{i\in I_r} u_i u_i^*\end{bmatrix}
\]
has norm bounded by $(1 + \sqrt{\epsilon})^2$. This implies that each block has norm bounded as in \eqref{mss5}.
\end{proof}

The following result shows that Lemma \ref{MSS2} also holds in infinite dimensional setting.

\begin{theorem}\label{MSS}
Let $I$ be at most countable index set. Let $\{\phi_i\}_{i \in I}$ be a Parseval frame in a separable Hilbert space $\mathcal H$,
\begin{equation}\label{mss1}
\sum_{i\in I}\phi_i \otimes \phi_i = \mathbf I \qquad\text{and}\qquad\|\phi_i\|^2 \le \delta
\quad\text{for all }i.
\end{equation}
Let $r\in\N$ and $t_1, \ldots, t_r>0$ satisfy $\sum_{k=1}^r t_k=1$. Then, there exists a partition $\{I_1,\ldots, I_r\}$ of $I$ such that 
\begin{equation}\label{mss2}
\bigg\| \sum_{i\in I_k} \phi_i \otimes \phi_i  \bigg\| \le  t_k (1 + \sqrt{r\delta})^2
\qquad\text{for all } k=1,\ldots,r. \end{equation}
\end{theorem}

\begin{proof}
First, observe that the Parseval frame assumption \eqref{mss3} can be weakened by the Bessel condition. Indeed, suppose that $\{u_i\}_{i\in [m]}$ is merely a Bessel sequence with bound $1$ and $||u_i||^2 \le \delta$. Define $d\times d$ matrix $T$ as
\[
T = \mathbf I - \sum_{i=1}^m u_i \otimes u_i.
\]
Since $T$ is positive semidefinite, we can find vectors $\{u_i\}_{i=m+1}^{m'}$, $m'>m$, such that
\[
T = \sum_{i=m+1}^{m'} u_i \otimes u_i \qquad\text{and}\qquad ||u_i||^2 \le \delta \text{ for }i\ge m+1.
\]
Indeed, it suffices to choose vectors $u_i$ to be appropriately scaled eigenvectors of $T$. Consequently, $\{u_i\}_{i\in [m']}$ becomes a Parseval frame for $\C^d$ and by Lemma \ref{MSS2} we can find a partition $\{I_1,\ldots, I_r\}$ of $[m']$ such that corresponding subsets $\{u_i\}_{i\in I_k}$ have required Bessel bounds. Restricting this partition to $[m]$ yields the same conclusion for $\{u_i\}_{i\in I_k\cap [m]}$, $k=1,\ldots,r$.

Now suppose $\{\phi_i\}_{i\in I}$ is a Parseval frame in an infinite dimensional Hilbert space $\mathcal H$ as in \eqref{mss1}. Since $\mathcal H$ is separable, $I$ is countable, and we may assume $I=\N$. For any $n\in \N$ we can apply Lemma \ref{MSS2} to the initial sequence $\{\phi_i\}_{i\in [n]}$. Hence, for each $n\in \N$ we have a partition $\{I_1^n, \ldots, I_r^n\}$ of $[n]$, which yields the required bound \eqref{mss2}. To show the existence of a global partition of $\{I_1,\ldots,I_r\}$ of $\N$ satisfying \eqref{mss2}, it suffices to apply the pinball principle \cite[Proposition 2.1]{CCLV}. This boils down to repeated applications of pigeonhole principle. The first vector $\phi_1$ must land infinitely many times to one of the slots $I_{j_1}^n$ for some $j_1=1,\ldots,r$. Let $N_1 \subset \N$ be the collection of all such $n$. Then, we repeat the same argument to the second vector $\phi_2$ for partitions of $[n]$, where $n\in N_1$. Again, we can find a slot $I_{j_2}^n$, where the second vector $u_2$ lands for infinitely many $n\in N_2 \subset N_1$. Repeating this process yields a nested sequence of infinite subsets $N_1 \supset N_2 \supset \ldots$ and indices $j_1,j_2,\ldots$ in $[r]$ such that the initial vectors $\phi_1,\ldots, \phi_m$, $m\in \N$, all land to the same respective slots $I^n_{j_1},\ldots,I^n_{j_m}$ for all $n\in N_m$. Define a global partition of $\N$ by $I_k =\{ i \in \N: j_i=k\}$, $k\in [r]$. Thus, \eqref{mss2} holds when $I_k$ replaced by $I_k \cap [m]$. Letting $m\to \infty$ shows the required Bessel bound \eqref{mss2}.
\end{proof}

As a corollary we obtain an infinite dimensional variant of \cite[Corollary 2.2]{AW}.

\begin{corollary}\label{ia} Under the same hypotheses as Theorem \ref{MSS}, there exists a partition $\{I_k\}_{k\in [r]}$ of $I$ such that
\begin{equation}\label{ia1}
\bigg\| \sum_{i \in I_k}  \phi_i \otimes \phi_i - t_k \mathbf I \bigg\| \le  2\sqrt{r\delta} +r\delta \qquad\text{for all } k=1,\ldots,r.
\end{equation}
\end{corollary}

\begin{proof}
Theorem \ref{MSS} yields
\begin{equation}\label{ia2}
\sum_{i\in I_k} \phi_i \otimes \phi_i  \le  t_k (1 + \sqrt{r\delta})^2 \mathbf I= t_k + t_k(2\sqrt{r\delta} +r\delta) \mathbf I.
\end{equation}
Summing the above over all $k' \ne k$ yields
\[
\mathbf I - \sum_{i\in I_k} \phi_i \otimes \phi_i = \sum_{i \in I \setminus I_k} \phi_i \otimes \phi_i \le \sum_{k' \ne k}  t_{k'} (1 + \sqrt{r\delta})^2 \mathbf I = (1-t_k) (1 + \sqrt{r\delta})^2 \mathbf I.
\]
Hence,
\begin{equation}\label{ia3}
\sum_{i\in I_k} \phi_i \otimes \phi_i \ge (1 - (1-t_k)  (1 + \sqrt{r\delta})^2) \mathbf I = (t_k - (1-t_k)(2\sqrt{r\delta} +r\delta)) \mathbf I.
\end{equation}
Combining \eqref{ia2} and \eqref{ia3} yields \eqref{ia1}.
\end{proof}

The next step is the following lemma due to Akemann and Weaver \cite[Lemma 2.3]{AW} which relaxes the assumption of Parseval frame by Bessel sequence.

\begin{lemma}
\label{awl}
There exists a universal constant $C>0$ such that the following holds.
Suppose $\{\phi_i\}_{i\in I}$ is a Bessel family with bound $1$ in a separable Hilbert space $\mathcal{H}$, which consists of vectors of norms $\|\phi_i\|^2\leq \ve$, where $\ve>0$. Let $S$ be its frame operator. Then for any $0 \le t \le 1$, there exists a subset $I_0 \subset I$ such that
\[
\bigg\| \sum_{i\in I_0} \phi_i \otimes \phi_i - t S \bigg\| \le C \ve^{1/4}.
\]
\end{lemma}

\begin{proof}
Let $S= \sum_{i\in I} \phi_i \otimes \phi_i$ be the frame operator of $\{\phi_i\}_{i\in I}$.
Assume momentarily that $\{\phi_i\}_{i\in I}$ is a Parseval frame. Applying Corollary \ref{ia} for $r=2$, $t_1=t$ and $t_2=1-t$ yields a subset $I' \subset I$ such that 
\begin{equation}\label{aw3}
\bigg\| \sum_{i \in I'} \phi_i \otimes \phi_i - t \mathbf I \bigg\| \le  2\sqrt{2\epsilon} +2\epsilon = O(\sqrt{\epsilon}).
\end{equation}
Here and in what follows we use big $O$ notation since we do not aim at controlling concrete constants.

In general, we use functional calculus to reduce the problem to the above case. That is, we define a projection $P=\ch_{[\sqrt{\epsilon},1]}(S)$, which ``ignores'' a non-invertible part of $S$. Let $\mathcal K$ be the range of $P$. Define an operator $B$ by $B:=S^{-1/2} P$. This operator is well-defined since $S$ is invertible on the range of $P$. Since $\sqrt{\epsilon} P \le SP \le P$, we have
\begin{equation}\label{aw4}
P \le B \le \epsilon^{-1/4} P.
\end{equation}
Define a family of vectors $\{\psi_i\}_{i\in I}$ by 
$\psi_i=B\phi_i$, $ i\in I$.
Since $P$ and $S^{1/2}$ commute,
\[
\sum_{i \in I} \psi_i \otimes \psi_i = B \bigg( \sum_{i\in I} \phi_i \otimes \phi_i \bigg) B = B S B =  S^{-1/2} P S S^{-1/2} P= P.
\]
By \eqref{aw4}, 
\[
||\psi_i||^2 \le \epsilon^{-1/2} ||P\phi_i||^2  \le \sqrt{\epsilon}.
\]
Thus, we can apply \eqref{aw3} to deduce the existence of a subset 
$I' \subset I$ such that 
\begin{equation}\label{aw6}
\bigg\| \sum_{i \in I'} \psi_i \otimes \psi_i - t P \bigg\| = O(\epsilon^{1/4}).
\end{equation}
We claim that
\begin{equation}\label{aw7}
\bigg\| \sum_{i \in I'} \phi_i \otimes \phi_i - t S \bigg\| = O(\epsilon^{1/4}).
\end{equation}
Indeed, let $D=\sum_{i \in I'} \phi_i \otimes \phi_i$. Then,
\begin{equation}
\label{aw8}
\begin{aligned}
\|
P ( D- t S ) P \| &= 
\bigg\| S^{1/2} \bigg( \sum_{i \in I'} \psi_i \otimes \psi_i - t P \bigg) S^{1/2} \bigg\|
\\
&\le \bigg\| \sum_{i \in I'} \psi_i \otimes \psi_i - t P \bigg\| = O(\epsilon^{1/4}).
\end{aligned}
\end{equation}
Since 
\[
0 \le D \le S \qquad\text{and}\qquad 0 \le S(\mathbf I-P) \le \sqrt{\epsilon}\mathbf I,
\]
we have for any $u \in\mathcal K^\perp$ and $v\in\mathcal H$,  
\[
|\langle D u, v \rangle |  \le \langle D^{1/2} u, D^{1/2} v \rangle | \le ||D^{1/2} u|| ||v || \le \langle S u,u \rangle ||v|| \le \sqrt{\epsilon}||u|| ||v||. 
\]
Thus,
\[
|| D (\mathbf I -P) || = ||(\mathbf I - P) D || \le \sqrt{\epsilon}.
\]
Since $ P S  (\mathbf I -P)  = (\mathbf I - P) S P = \mathbf 0$, by \eqref{aw8} the norm $\| D-t S \|$ is less than
\begin{align*}
 & \le  \|P ( D- t S ) P \| + 2 \|(\mathbf I -P) ( D- t S ) P \| + 
\|(\mathbf I -P) ( D- t S ) (\mathbf I -P) \| 
\\
& \le O(\epsilon^{1/4}) + 2  \|(\mathbf I -P)  D P \|  +  \|(\mathbf I -P) D  (\mathbf I -P) \| + \|(\mathbf I -P) S  (\mathbf I -P) \| 
\\
&= O(\epsilon^{1/4})+O(\epsilon^{1/2})=O(\epsilon^{1/4}). 
\end{align*}
This proves the claim and completes the proof of the lemma.
\end{proof}

We are now ready to prove an infinite dimensional formulation of approximate Lyapunov theorem for discrete frames due to Akemann and Weaver \cite[Theorem 2.4]{AW}.

\begin{theorem}
\label{aw}
There exists a universal constant $C_0>0$ such that the following holds.
Suppose $\{\phi_i\}_{i\in I}$ is a Bessel family with bound $1$ in a separable Hilbert space $\mathcal{H}$, which consists of vectors of norms $\|\phi_i\|^2\leq \ve$, where $\ve>0$. Suppose that $0\le t_i \le 1$ for all $i\in I$. Then, there exists a subset $I_0 \subset I$ such that
\begin{equation}\label{aw0}
\bigg\| \sum_{i\in I_0} \phi_i \otimes \phi_i - \sum_{i\in I} t_i \phi_i \otimes \phi_i \bigg\| \le C_0  \ve^{1/8}.
\end{equation}
\end{theorem}

\begin{proof}
We proceed exactly as in the proof of \cite[Theorem 2.4]{AW}. That is, we take $n=\lfloor \epsilon^{-1/8} \rfloor$ and we partition $I$ into subsets 
\[
I_k = \{ i \in I: (k-1)/n< t_i \le k/n \}, \qquad k=1,\ldots,n.
\]
Then, we apply \eqref{aw7} for each family $\{\phi_i\}_{i\in I_k}$ for $t=k/n$ to find subsets $I_k' \subset I_k$ such that
\[
\bigg\| \sum_{i \in I'_k} \phi_i \otimes \phi_i - \frac{k}{n}  \sum_{i\in I_k} \phi_i \otimes \phi_i \bigg\| = O(\epsilon^{1/4}).
\]
Taking $I_0 = \bigcup_{k=1}^n I'_k$, we have
\begin{align*}
& \bigg\| \sum_{i \in I_0} \phi_i \otimes \phi_i  -  \sum_{i\in I} t_i \phi_i \otimes \phi_i \bigg\| 
\\
& \le 
 \bigg\| \sum_{k=1}^n  \bigg(\sum_{i \in I'_k} \phi_i \otimes \phi_i- \frac{k}{n}  \sum_{i\in I_k} \phi_i \otimes \phi_i\bigg) \bigg\| +  \bigg\| \sum_{k=1}^n  \sum_{i\in I_k} (k/n-t_i)\phi_i\otimes \phi_i \bigg\|
  \\
& \le    \sum_{k=1}^n  \bigg\| \sum_{i \in I'_k} \phi_i\otimes \phi_i - \frac{k}{n}  \sum_{i\in I_k} \phi_i \otimes \phi_i \bigg\| + O(\epsilon^{1/8}) ||S|| 
\\
& \le n O(\epsilon^{1/4})  + O(\epsilon^{1/8}) =  O(\epsilon^{1/8}).
\end{align*}
This proves Theorem \ref{aw}.
\end{proof}

 As a corollary we obtain a discrete analogue of Theorem \ref{lyu}.

\begin{corollary}\label{dlp}
Suppose $\{\phi_i\}_{i\in I}$ is a Bessel family with bound $1$ in a separable Hilbert space $\mathcal{H}$, which consists of vectors of norms $\|\phi_i\|^2\leq \ve$, where $\ve>0$. Let $\mathcal S$ be the set of all partial frame operators
\[
\mathcal S = \bigg\{ \sum_{i \in I'} \phi_i \otimes \phi_i : I' \subset I \bigg\}.
\]
Then $\mathcal S$ is an approximately convex subset of  $\mathcal B(\mathcal H)$. More precisely, for every $T$ in the convex hull of $\mathcal S$,  there exists $S \in \mathcal S$ such that $||S-T|| \le C_0 \ve^{1/8}$.
\end{corollary}

The assumption that $\{\phi_i\}_{i\in I}$ has a Bessel bound $1$ is not essential. Indeed, a scaling of Corollary \ref{dlp} for Bessel sequences with an arbitrary bound $B$ yields the estimate \eqref{sb}.
Finally, we can combine Theorem \ref{lyu} and Corollary \ref{dlp} to obtain Lyapunov's theorem for continuous frames on general measure spaces. This is due to the fact that every measure space decomposes into its atomic and non-atomic components and a continuous frame on an atomic measure space coincides with a discrete frame. 

\begin{corollary}
Suppose that $\{\phi_t\}_{t\in X}$ is a continuous Bessel family in $\mathcal H$ with bound $B$ on any measure space $(X,\mu)$. Let $\mathcal S$ be the set of all partial frame operators as in \eqref{lyu1}.
Define
\[
\ve_0 = \sup \{ \mu(E) ||\phi_t||^2:  E\text{ is an atom in } X \text{ and } t\in E \}.
\]
If $X$ is non-atomic, then we take $\ve_0=0$. Then, $\mathcal S$ is an approximately convex subset of  $\mathcal B(\mathcal H)$. More precisely, for every $T$ in the convex hull of $\mathcal S$ and for every $\ve>\ve_0$, there exists $S \in \mathcal S$ such that 
\begin{equation}\label{sb}
||S-T|| \le C_0 B^{7/8} \ve^{1/8}.
\end{equation}
\end{corollary}

\section{Scalable frames and discretization problem} \label{S4}

In this section we present the solution of the discretization problem due to Freeman and Speegle \cite{FS}. The key result in the proof is a sampling theorem for scalable frames. The proof is a technical and brute force application of the following result on frame partitions, see \cite[Theorem 1.7]{FS} and \cite[Lemma 2]{NOU}. Our aim is outline the essential parts of this argument.

\begin{theorem}\label{fp}
There exists constants $A_0,B_0>0$ such the following holds. Every tight frame of vectors in the unit ball of a separable Hilbert space $\mathcal H$ with frame constant $\ge 1$ can be partitioned into a collection of frames of $\mathcal H$ each with lower and upper frame bounds $A_0$ and $B_0$.
\end{theorem}

Following \cite{NOU, OU}, we will need two lemmas in the proof of Theorem \ref{fp}.

\begin{lemma}\label{ou}
Let $I$ be at most countable index set and let $\mathcal H$ be a separable Hilbert space. Let $\{\phi_i\}_{i \in I}$ be a frame with bounds $A$ and $B$,
\begin{equation}\label{ou1}
A \mathbf I \le \sum_{i\in I} \phi_i \otimes \phi_i \le B \mathbf I \qquad\text{and}\qquad\|\phi_i\|^2 \le \delta
\quad\text{for all }i.
\end{equation}
If $A>\delta$, then there exists a partition of $I$ into subsets $I_1$ and $I_2$ such that for $k=1,2$,
\begin{equation}\label{ou2}
\frac{1-5 \sqrt{\delta/A}}2 A \mathbf I \le \sum_{i\in I_k} \phi_i \otimes \phi_i\le \frac{1+5 \sqrt{\delta/A}}2 B \mathbf I.
\end{equation}
\end{lemma}

\begin{proof}
If $\{\phi_i\}_{i \in I}$ is a Parseval frame and $\delta<1$, then by Theorem \ref{MSS} for $t_1=t_2=1/2$,  we have a partition so that for $k=1,2$,
\[
\sum_{i\in I_k} \phi_i \otimes \phi_i \le \frac{(1+\sqrt{2\delta})^2}2 \mathbf I \le \frac{1+5\sqrt{\delta}}2 \mathbf I.
\]
Since 
\[
\mathbf I - \sum_{i\in I_1} \phi_i \otimes \phi_i  = \sum_{i\in I_2} \phi_i \otimes \phi_i ,
\]
we obtain two-sided estimate \eqref{ou2} in the special case $A=B=1$.

If $\{\phi_i\}_{i \in I}$ is a general frame, then let $S$ be its frame operator.  Note that $A \mathbf I \le S \le B \mathbf I$ and hence $B^{-1} \mathbf I \le S^{-1} \le A^{-1} \mathbf I$. Hence, $\{S^{-1/2}\phi_i\}_{i \in I}$ is a Parseval frame and 
\[
||S^{-1/2}\phi_i||^2 \le A^{-1} ||\phi_i||^2 \le \delta/A.
\]
Hence, we can apply the Parseval frame case of \eqref{ou2} and
\[
\sum_{i\in I_k} S^{-1/2} \phi_i  \otimes S^{-1/2}\phi_i = S^{-1/2} \bigg( \sum_{i\in I_k}  \phi_i  \otimes \phi_i \bigg)S^{-1/2}
\]
to deduce \eqref{ou2}.
\end{proof}

\begin{lemma}\label{sou}
Let $0<\delta<1/100$. Define sequences $\{A_j\}_{j=0}^\infty$ and $\{B_j\}_{j=0}^\infty$ inductively by 
\[
A_0=B_0=1,\qquad A_{j+1}=A_j \frac{1-5 \sqrt{\delta/A_j}}2, \qquad B_{j+1}=B_j \frac{1+5 \sqrt{\delta/B_j}}2.
\]
Then, there exists an absolute constant $C$ and an integer $L\ge 0$ such that
\begin{equation}\label{sou2}
A_j \ge 100\delta \text{ for }j \le L, \quad
25 \delta \le A_{L+1} <100\delta, \quad B_{L+1} < C A_{L+1}.
\end{equation}
\end{lemma}

\begin{proof}
If $A_j \ge 100 \delta$, then
\[
\frac{A_j}4 \le A_{j+1} \le \frac{A_j}2.
\]
Let $L \ge 1$ be the largest integer such that $A_L \ge 100 \delta$. For $j\le L$, let $C_j= 5 \sqrt{\delta/A_j}$. Note that $C_{L-j} < 2^{-1-j/2}$ for $j=0,\ldots,L$. Hence, by telescoping
\[
\frac{B_{L+1}}{A_{L+1}} = \prod_{j=0}^L \frac{1+C_j}{1-C_j} < C:= \prod_{j=0}^\infty \frac{1+2^{-1-j/2}}{1-2^{-1-j/2}}<\infty.
\]
This proves \eqref{sou2}.
\end{proof}

Now we are ready to prove Theorem \ref{fp}.

\begin{proof}
Suppose $\{\phi_i\}_{i\in I}$ is a tight frame with constant $K\ge 1$ such that $||\phi_i|| \le 1$ for all $i\in I$. Hence, $\psi_i = K^{-1/2} \phi_i$, $i\in I$, is a Parseval frame such that $|| \psi_i||^2 \le \delta:= 1/K$.

We shall apply Lemma \ref{ou} recursively. If $100\delta<1$, then we apply Lemma \ref{ou} to split it it into two frames $\{\psi_i\}_{i\in I_k}$, $k=1,2$, with bounds $A_1$ and $B_1$. If $100\delta<A_1$, then we apply Lemma \ref{ou} again to each frame $\{\psi_i\}_{i\in I_k}$; otherwise we stop. Let $L\ge 0$ be the stopping time from Lemma \ref{sou}. We continue applying Lemma \ref{ou} to produce a partition of $\{\psi_i\}_{i\in I}$ into $2^{L+1}$ frames with bounds $A_{L+1}$ and $B_{L+1}$. This corresponds to a partition of $\{\phi_i\}_{i\in I}$ into $2^{L+1}$ frames with bounds $A_{L+1}/\delta$ and $B_{L+1}/\delta$. By \eqref{sou2}, these bounds satisfy
 \[
 25 \le A_{L+1}/\delta, \qquad  B_{L+1}/\delta< C A_{L+1}/\delta \le 100C.
 \]
 
 If $100\delta>1$, then there is no need to apply the above procedure since $\{\phi_i\}_{i\in I}$ is a tight frame with bound $1/\delta$. Hence, it is trivially a frame with bounds $1$ and $100$. Consequently, every tight frame with constant $\ge 1$ can be partitioned into frames with bounds $1$ and $100C$.
\end{proof}

By scaling we can deduce a variant of Theorem \ref{fp} for arbitrary frames.

\begin{corollary}\label{cfp}
Let $0<N \le A \le B<\infty$. Let $\{\phi_i\}_{i\in I}$ be a frame with bounds $A$ and $B$ in a separable infinite-dimensional Hilbert space $\mathcal H$ with norms $||\phi_i||^2 \le N$ for all $i\in I$. Then, there exists a partition $I_1,\ldots, I_r$ of $I$ such that for every $k=1,\ldots,r$, $ \{\phi_i \}_{i\in I_k}$ is a frame with bounds
\begin{equation}\label{fs1}
A_0 N \qquad\text{and}\qquad  B_0 N\frac{B}A,
\end{equation}
where $A_0$ and $B_0$ are constants from Theorem \ref{fp}.
\end{corollary}

\begin{proof}
Let $S$ be the frame operator of $\{\phi_i \}_{i\in I}$. Note that $A \mathbf I \le S \le B \mathbf I$ and hence $B^{-1} \mathbf I \le S^{-1} \le A^{-1} \mathbf I$. Hence, $\{S^{-1/2} \phi_i \}_{i\in I}$ is a Parseval frame with norms $||S^{-1/2} \phi_i ||^2 \le N/A$. Thus, we can apply Theorem \ref{fp} to a tight frame $\{(A/N)^{1/2} S^{-1/2} \phi_i \}_{i\in I}$ with frame constant $A/N\ge 1$, which consists of vectors in the unit ball of $\mathcal H$. That is, there exists a partition $I_1,\ldots, I_r$ of $I$ such that for every $k=1,\ldots,r$, $\{(A/N)^{1/2} S^{-1/2} \phi_i \}_{i\in I_k}$ is a frame with bounds $A_0$ and $B_0$. Therefore, for any $\phi \in \mathcal H$,
\[
\begin{aligned}
A_0 N ||\phi||^2
&\le \frac{A_0N}A ||S^{1/2}\phi||^2
 \le \sum_{i\in I_k} |\lan (N/A)^{1/2}S^{1/2} \phi, (A/N)^{1/2} S^{-1/2} \phi_i \ran|^2  \\
& = \sum_{i\in I_k}  |\lan \phi, \phi_i \ran|^2
  \le \frac{B_0N}A ||S^{1/2}\phi||^2 \le \frac{B_0BN}A ||\phi||^2.
\end{aligned}
\]
\end{proof}

Recall that $\{\phi_i\}_{i\in I}$ in $\mathcal H$ is a scalable frame if there exists a sequence of scalars $\{a_i\}_{i\in I}$ such that $\{a_i \phi_i \}_{i\in I}$ is a Parseval frame. Using Theorem \ref{fp} Freeman and Speegle \cite{FS} have derived the following sampling theorem for scalable frames. 

Classically, a sampling process describes a procedure of choosing points from a given set where every point is chosen at most once. In contrast, the sampling function $\pi: \N \to I$ in Theorem \ref{atb} is in general {\bf not} injective.

\begin{theorem}\label{atb}
There exist universal constants $A_0,B_0>0$ such that the following holds.
Let $\{\phi_i\}_{i\in I}$ be a scalable frame in a separable Hilbert space $\mathcal H$ with norms $||\phi_i||^2 \le 1$ for all $i\in I$. Then, for any $0<\ve<1$, there exists a sampling function $\pi:\N \to I$ such that $\{\phi_{\pi(n)}\}_{n\in \N}$ is a frame with bounds $A_0(1-\ve)$ and $2B_0(1+\ve)$.
\end{theorem}

\begin{remark}
The role of $\ve>0$ in the formulation of Theorem \ref{atb} is not essential. For example, taking $\ve=1/2$ yields a frame $\{\phi_{\pi(n)}\}_{n\in \N}$ with bounds $A_0/2$ and $3B_0$. Here, $A_0$ and $B_0$ are the same constants as in Theorem \ref{fp}. Hence, the above formulation merely reflects the explicit dependence of frame bounds on these constants as in \cite{FS}.
\end{remark}

\begin{proof}[finite dimensional case]
It is instructive to show Theorem \ref{atb} for a finite dimensional space $\mathcal H$ first. In this case a sampling function $\pi$ is defined on a finite subset of $\N$. Choose a finite subset $I' \subset I$ such that $\{a_i \phi_i \}_{i\in I'}$ is a frame with bounds $1-\ve/2$ and $1$, and $a_i \ne 0$ for all $i\in I'$. Our goal is to reduce to the case when all coefficients $a_i$ are approximately equal.

Let $\eta = \inf_{i\in I'} |a_i|^2>0$. Let $K\in \N$ be a parameter (to be determined later). Then, we replace each element $a_i \phi_i$, $i \in I'$, by a finite collection of vectors
\[
\genfrac{}{}{0pt}{0}{\underbrace{\frac{a_i}{\sqrt{N_i}} \phi_i, \ldots, \frac{a_i}{\sqrt{N_i}} \phi_i}}{N_i} \qquad\text{where } N_i= \lceil K |a_i|^2/\eta \rceil.
\]
More precisely, let $N= \sum_{i\in I'} N_i$ and let $\kappa: [N] \to I'$ be a mapping such that each value $i\in I'$ is taken precisely $N_i$ times.
This yields a new collection of vectors $\{\phi_{\kappa(n)}\}_{n\in [N]}$ in which each vector $\phi_i$ is repeated $N_i$ times and a corresponding sequence $\{b_n\}_{n\in [N]}$, where $b_n=a_{\kappa(n)}/\sqrt{N_{\kappa(n)}}$. By our construction, we have  
\begin{equation}\label{atb2}
\frac{\eta}{K+1} \le \sup_{n\in [N]} |b_n|^2 \le \frac\eta K.
\end{equation}
Moreover, the frame operator corresponding to $\{a_i \phi_i \}_{i\in I'}$ is the same as the frame operator of $\{b_n \phi_{\kappa(n)} \}_{n\in [N]}$. We shall apply Corollary \ref{cfp} to a frame $\{\sqrt{\eta/K} \phi_{\kappa(n)} \}_{n\in [N]}$. By \eqref{atb2}, its frame bounds are given by
\[
\begin{aligned}
(1-\ve/2)\mathbf I 
& \le \sum_{n\in [N]} |b_n|^2 \phi_{\kappa(n)}  \otimes \phi_{\kappa(n)} 
\le
\frac{\eta}K \sum_{n\in [N]} \phi_{\kappa(n)} \otimes  \phi_{\kappa(n)} 
\\
&\le \frac{K+1}K \sum_{n\in [N]} |b_n|^2 \phi_{\kappa(n)} \otimes \phi_{\kappa(n)} \le  \frac{K+1}K \mathbf I.
\end{aligned}
\]
If $K\in \N$ satisfies $\eta/K \le 1-\ve/2$, then Corollary \ref{cfp} yields a partition of $[N]$ into subsets $I_1,\ldots,I_r$ such that each $\{\sqrt{\eta/K} \phi_{\kappa(n)} \}_{n\in I_k}$, $k=1,\ldots, r$, is a frame with bounds
\[
A_0 \frac{\eta}K \qquad\text{and}\qquad  B_0 \frac{\eta}K \frac{K+1}{K(1-\ve/2)}.
\]
Now choose $K\in \N$ large enough so that $\frac{K+1}{K(1-\ve/2)} \le 1+\ve$. Consequently, each collection $\{\phi_{\kappa(n)} \}_{n\in I_k}$, $k=1,\ldots, r$, is a frame with bounds $A_0$ and $B_0(1+\ve)$. Hence, the mapping $\pi: I_1 \to I$ given by restricting $\kappa$ to $I_1$ is the required sampling function.
\end{proof}

Note that in the finite dimensional case we have obtained a better frame upper bound and we have not used the full strength of Corollary \ref{cfp}. The proof of the infinite dimensional case of Theorem \ref{atb} is quite involved and technical. Hence, we only present its main steps.

\begin{proof}[outline of the infinite dimensional case]
Suppose $\{\phi_i\}_{i\in I}$ is a scalable frame and $\{a_i\}_{i\in I}$ is the corresponding sequence of coefficients such that $\{a_i \phi_i\}_{i\in I}$ is a Parseval frame in $\mathcal H$. Since $\mathcal H$ is infinite dimensional, we may assume that $I=\N$ and all vectors $\phi_i$ are non-zero.

Let $\{\ve_k\}_{k\in\N}$ be a sequence of positive numbers (to determined later). We shall construct a sequence of orthogonal finite dimensional spaces $\{\mathcal H_k\}_{k\in \N}$ such that $\bigoplus_{k\in \N} \mathcal H_k = \mathcal H$, and an increasing sequence of natural numbers $\{K_i\}_{i\in \N}$ by the following inductive procedure. Let $\mathcal H_1=\{0\}$ be the trivial space and $K_1=1$.  Assume we have already constructed subspaces $\mathcal H_1,\ldots,\mathcal H_n$ and natural numbers $K_1,\ldots,K_{n}$, $n\ge 1$. Then,
\begin{itemize}
\item define a subspace
\begin{equation}\label{hn}
\mathcal H_{n+1} = \spa \{ P_{(\mathcal H_1  \oplus \ldots \oplus \mathcal H_n )^\perp} \phi_i : 1 \le i \le K_n\},
\end{equation}
\item choose $K_{n+1}>K_n\in \N$ large enough so that
\begin{equation}\label{in-0}
\{ a_i P_{\mathcal H_1 \oplus \ldots \oplus \mathcal H_{n+1}} \phi_i\}_{i>K_{n+1}} \text{ is Bessel with bound } \ve_{n+1},
\end{equation}
\item and repeat the above process ad infinitum.
\end{itemize}

Since $\{a_i P_{\mathcal H_1 \oplus \ldots \oplus \mathcal H_n} \phi_i\}_{i\in \N}$ is a Parseval frame in 
$\mathcal H_1 \oplus \ldots \oplus \mathcal H_n$, by \eqref{in-0} for any $n\in \N$ we have
\begin{multline}\label{in0}
\{ a_i P_{\mathcal H_1 \oplus \ldots \oplus \mathcal H_n} \phi_i\}_{i=1}^{K_{n}} \text{ is a frame in }\mathcal H_1 \oplus \ldots \oplus \mathcal H_n 
\\
\text{ with bounds } 1- \ve_n \text{ and } 1.
\end{multline}
By \eqref{hn}
\begin{equation}\label{in1}
\phi_i \in \mathcal H_1 \oplus \ldots \oplus \mathcal H_{n+1} \qquad\text{for all } i =1,\ldots,K_n.
\end{equation}
Thus, by \eqref{in0}, for any $1\le m\le n\in \N$,
\begin{multline}\label{in2}
\{ a_i P_{ \mathcal H_m \oplus \ldots \oplus \mathcal H_n} \phi_i\}_{i=K_{m-2}+1}^{K_n} \text{ is a frame in }\mathcal H_m \oplus \ldots \oplus \mathcal H_n
\\
 \text{ with bounds } 1- \ve_n \text{ and } 1.
\end{multline}
Here, we use the convention that $K_{-1}=K_0=0$. 

The spaces $\{\mathcal H_k\}_{k\in \N}$ are building blocks in constructing a sampling frame. First we group these spaces into blocks with overlaps
$\bigoplus_{k=M_r}^{N_r} \mathcal H_k$ for appropriate increasing sequences $\{M_r\}_{r\in \N}$ and $\{N_r\}_{r\in \N}$ of integers with $M_1=1$ such that  consecutive intervals $[M_r,N_r]$ and $[M_{r+1},N_{r+1}]$ have significant overlaps, but intervals $[M_r,N_r]$ and $[M_{r+2},N_{r+2}]$ are disjoint. An elaborate argument using \eqref{in2} shows the existence of a sampling function $\pi_r$, $r\ge 1$, defined on a finite set $I_r$ with values in $(K_{M_r-2},K_{N_r}]$ such that
\begin{multline}
\{  P_{\mathcal H_{M_r} \oplus \ldots \oplus \mathcal H_{N_r}} \phi_{\pi_r(i)} \}_{i\in I_r} \text{ is a frame in }\mathcal H_{M_r} \oplus \ldots \oplus \mathcal H_{N_r}
\\
 \text{ with bounds } A_0 \text{ and } B_0(1+\ve).
\end{multline}
This part uses Corollary \ref{cfp} in an essential way as in the proof of finite dimensional case of Theorem \ref{atb}. Moreover, for appropriate choice of a sequence $\{\ve_k\}_{k\in\N}$, one can control the interaction between consecutive blocks $\bigoplus_{k=M_r}^{N_r} \mathcal H_k$ to deduce that the vectors $\{\phi_{\pi_r(i)} \}_{i\in I_r}$ do not interfere too much beyond these blocks. 
 Hence, roughly speaking $\{\phi_{\pi_r(i)} \}_{i\in I_r}$ forms a frame in $\bigoplus_{k=M_r}^{N_r} \mathcal H_k$ with bounds  $A_0(1-\ve) \text{ and } B_0(1+\ve)$. The hardest and most technical part is showing the lower frame bound which necessitates sufficiently large overlaps between consecutive intervals $[M_r,N_r]$ and $[M_{r+1},N_{r+1}]$. 

Now it remains to put these frames together by defining a global sampling function $\pi$ defined on a disjoint union $I_\infty=\bigcup_{r\in\N} I_r$ by $\pi(i)=\pi_r(i)$ if $i\in I_r$. Due to overlaps the upper frame bound of $\{\phi_{\pi(i)}\}_{i\in I_\infty}$ bumps to $2B_0(1+\ve)$ with the lower bound staying the same at $A_0(1-\ve)$. This completes an outline of the proof of Theorem \ref{atb}.
\end{proof}

By scaling Theorem \ref{atb} we obtain the following corollary. The proof of Corollary \ref{fs} mimics that 
of Corollary \ref{cfp}.

\begin{corollary}\label{fs}
Let $0<A \le B<\infty$ and $N>0$. Let $\{\phi_i\}_{i\in I}$ be a sequence of vectors in a separable infinite dimensional Hilbert space $\mathcal H$ with norms $||\phi_i||^2 \le N$ for all $i\in I$. Suppose there exists scalars $\{a_i\}_{i\in I}$ such that $\{a_i \phi_i \}_{i\in I}$ is a frame with bounds $A$ and $B$. Then, for any $\ve>0$, there exists a sampling function $\pi:\N \to I$ such that $\{\phi_{\pi(n)}\}_{n\in \N}$ is a frame with bounds $A_0N(1-\ve)$ and $2B_0N\frac{B}{A}(1+\ve)$.
\end{corollary}

\begin{proof}
Let $S$ be the frame operator of $\{a_i \phi_i \}_{i\in I}$. Then, $\{a_i S^{-1/2} \phi_i \}_{i\in I}$ is a Parseval frame and $||S^{-1/2} \phi_i ||^2 \le N/A$. Applying Theorem \ref{atb} for a scalable frame $\{ (A/N)^{1/2} S^{-1/2} \phi_i \}_{i\in I}$, which consists of vectors in the unit ball of $\mathcal H$, yields a sampling function $\pi: \N \to I$ such that $\{ (A/N)^{1/2} S^{-1/2} \phi_{\pi(n)} \}_{n\in \N}$ is a frame with bounds $A_0(1-\ve)$ and $2B_0(1+\ve)$. Therefore, for any $\phi \in \mathcal H$,
\[
\begin{aligned}
&A_0 N (1-\ve) ||\phi||^2
\le A_0 (1-\ve) \frac{N}{A} ||S^{1/2}\phi||^2
\\
& \le \sum_{n\in \N} |\lan (N/A)^{1/2}S^{1/2} \phi, (A/N)^{1/2} S^{-1/2} \phi_{\pi(n)} \ran|^2   = \sum_{n\in \N}  |\lan \phi, \phi_{\pi(n)} \ran|^2 
\\
&  \le 2B_0(1+\ve) \frac{N}{A} ||S^{1/2}\phi||^2 \le \frac{2B_0BN(1+\ve)}A ||\phi||^2.
\end{aligned}
\]
\end{proof}

We are now ready to prove the sampling theorem for bounded continuous frames due to Freeman and Speegle \cite[Theorem 5.7]{FS}.

\begin{theorem}\label{db} Let $(X,\mu)$ be a measure space and let $\mathcal H$ be a separable Hilbert space.
Suppose that $\{\phi_t\}_{t\in X}$ is a continuous frame in $\mathcal H$ with frame bounds $A$ and $B$, which is bounded by $N$, i.e., \eqref{cf2} holds. Then, there exists a sequence $\{t_n\}_{n\in I}$ in $X$, where $I\subset \N$, such that $\{\phi_{t_n} \}_{n\in I}$ is a frame with bounds $A_0N$ and $3B_0N\frac{B}{A}$, where $A_0$ and $B_0$ are constants from Theorem \ref{fp}.
\end{theorem}

\begin{proof}
We shall prove Theorem \ref{db} under the assumption that $\mathcal H$ is infinite dimensional and $I=\N$. A finite dimensional case is a simple modification of the following argument, where $I \subset \N$ is finite.

By Lemma \ref{approx} and Remark \ref{rem}, for every $\ve>0$, we can find a partition $\{X_n\}_{n\in\N}$ of $X$ and a sequence $\{t_n\}_{n\in \N}$ in  $X$ such that 
\begin{equation}\label{db1}
\{a_n \phi_{t_n}\}_{n\in \N},
\qquad\text{where } a_n = \sqrt{\mu(X_n)}
\end{equation}
is a frame with frame bounds $A(1-\ve)$ and $B(1+\ve)$. In a case when $a_n=\infty$, we necessarily have $\phi_{t_n}=0$, so we can simply ignore this term. Therefore, any continuous frame can be sampled by a scalable frame \eqref{db1} with nearly the same frame bounds. Note that the boundedness assumption \eqref{cf2} was not employed so far.

Next we apply Corollary \ref{fs} to the frame \eqref{db1} with a trivial norm bound  $||\phi_{t_n}||^2 \le N/(1-\ve)$. Hence, there exists a sampling function $\pi: \N \to \N$ such that $\{\phi_{t_{\pi(n)}}\}_{n\in \N}$ is a frame with bounds $A_0N$ and $2B_0N \frac{B(1+\ve)^2}{A(1-\ve)^2}$. Choosing sufficiently small $\ve>0$ shows that $\{t_{\pi(n)}\}_{n\in \N}$ is the required sampling sequence.
\end{proof}

The solution of the discretization problem by Freeman and Speegle \cite{FS} takes the following form.

\begin{theorem}
Let $X$ be a measurable space in which every singleton is measurable. Let $\mathcal H$ be a separable Hilbert space. Let $\phi: X \to \mathcal H$ be measurable. Then, the following are equivalent:
\begin{enumerate}[(i)]
\item 
there exists a sampling sequence $\{t_i\}_{i\in I}$ in $X$, where $I \subset \N$, such that $\{\phi_{t_i}\}_{i\in I}$ is a frame in $\mathcal H$,
\item there exists a positive, $\sigma$-finite measure $\nu$ on $X$ so that $\phi$ is a continuous frame in $\mathcal H$ with respect to $\nu$, which is bounded $\nu$-almost everywhere.
\end{enumerate}
\end{theorem}

\begin{proof}
The implication $\Leftarrow$ follows from Theorem \ref{db}. Now assume that there exists $\{t_i\}_{i\in I}$ such that $\{\phi_{t_i}\}_{i\in I}$ is a frame in $\mathcal H$. Since some points might be sampled multiple times, we need to define a counting measure
\[
\nu = \sum_{i\in I} \delta_{t_i},
\]
where $\delta_t$ denotes the point mass at $t\in X$. Since singletons are measurable, $\nu$ is a measure on $X$ and the frame operator of $\{\phi_t\}_{t\in X}$  with respect to $\nu$ is the same as the frame operator of $\{\phi_{t_i}\}_{i\in I}$. Thus, $\{\phi_t\}_{t\in X}$  is a continuous frame with respect to $\nu$, which is bounded on countable set $X'=\{t_i:i \in I\}$  and $\nu(X\setminus X')=0$.
\end{proof}

 \section{Examples}\label{S5}
 
In the final section we show applications of the discretization theorem from the previous section. We do not aim to show the most general results, but instead we illustrate Theorem \ref{db} for the well-known classes of continuous frames. While it might look surprising at first glance, we can also apply Theorem \ref{db} for discrete frames.

\begin{example}[Discrete frames]
Suppose that $\{\psi_n\}_{n\in\N}$ is a tight frame of vectors in the unit ball of a Hilbert space $\mathcal H$ with frame constant $K>0$. Let $\mu$ be the measure on $X=\N$ such that $\mu(\{n\})=1/K$ for all $n\in\N$. Hence, we can treat $\{\psi_n\}_{n\in X}$ as a continuous Parseval frame. Then, by Theorem \ref{db} there exists a sampling function $\kappa: \N \to \N$ such that $\{\psi_{\kappa(n)}\}_{n\in\N}$ is a frame with bounds $A_0$ and $3B_0$. Hence, we obtain a weak version of Theorem \ref{fp} on frame partitions.
\end{example}

A general setting for which Theorem \ref{db} applies involves continuous frames obtained by square integrable group representations. 

\begin{definition}
Let $G$ be locally compact group and let $\mu$ be the left Haar measure on $G$. Let $\pi: G \to \mathcal U(\mathcal H)$ be its unitary representation. We say that $\pi$ is a square integrable representation if there exits vectors $\psi^1,\ldots,\psi^n \in\mathcal H$ and constants $0<A \le B< \infty$ such that 
\begin{equation}\label{sr}
A||f||^2 \le \sum_{i=1}^n \int_G |\langle f, \pi(g)(\psi^i) \rangle|^2 d\mu (g) \le B ||f||^2 \qquad\text{for all }f\in\mathcal H.
\end{equation}
\end{definition}

For the sake of simplicity assume that $n=1$ in the above definition. Then, a square integrable representation defines a continuous frame on $(G,\mu)$ of the form $g \mapsto \pi(g)(\psi)$. The above definition encompasses three major examples of continuous frames: continuous Fourier frames,
continuous Gabor frames, and continuous wavelets.

\begin{example}[Fourier frames] Let $G=\R$ and let $\mu$ be the Lebesgue measure. Let $S \subset \R$ be a measurable subset of $\R$ of finite measure and let $\mathcal H =L^2(S)$. Define $\pi: \R \to L^2(S)$ 
\[
\pi(t)(\psi) = e^{2\pi i t \cdot } \psi, \qquad t\in \R, \psi\in L^2(S).
\]
Take $\psi=\ch_S$. Then, by the Plancherel theorem for any $f\in L^2(S)$,
\[
\int_{\R}  |\langle f, \pi(t)(\ch_S) \rangle|^2 dt = \int_{\R} \bigg| \int_S f(x) e^{-2\pi i t x} dx\bigg|^2 dt = \int_{\R} |\hat f(t)|^2 dt = ||f||^2.
\]
Here, we identify $L^2(S)$ with the subspace of $L^2(\R)$ of functions vanishing outside of $S$ and $\hat f$ is the Fourier transform of $f$. Thus, $\pi$ is a square integrable representation and $\{\phi_t:=e^{2\pi i t \cdot} \ch_S\}_{t\in\R}$ is a continuous Parseval frame in $L^2(S)$. By Theorem \ref{db} there exists a sampling sequence $\{t_n\}_{n\in\Z}$ such that $\{\phi_{t_n}\}_{n\in\Z}$ is a frame for $L^2(S)$ with bounds $A_0|S|$ and $3B_0 |S|$. 
This way we recover the result of Nitzan, Olevskii, and Ulanovskii \cite{NOU} on the sampling of continuous Fourier frames. 

\begin{theorem} For every set $S \subset \R$ of finite measure, there exists a discrete set of frequencies  $\Lambda \subset \R$ such that $\{e^{2\pi i x \lambda} \}_{\lambda \in \Lambda}$ is a frame in $L^2(S)$ with bounds $A \ge c|S|$ and $B \le C |S|$, where $c$ and $C$ are absolute constants.
\end{theorem} 

Note that Theorem \ref{db} does not guarantee in any way that the sampling set $\Lambda = \{ t_n: n\in\Z\}$ is discrete. However, we can invoke Beurling's density theorem for Fourier frames \cite[Lemma 10.25]{OU}. If $\{e^{2\pi i x \lambda} \}_{\lambda \in \Lambda}$ is a Bessel sequence with bound $C|S|$, then there exists a constant $K>0$ such that 
\[
\frac{ \#| \Lambda \cap \Omega|}{|\Omega|} \le 4C|S|,
\]
for every interval $\Omega \subset \R$ with length $|\Omega| \ge K$.
It is worth adding that there exists a continuous Fourier frame which does not admit a discretization by any regular grid, see \cite[Example 2.72]{Fu}.
\end{example}

\begin{example}[Gabor frames]
Let $G=\R^2$ and let $\mu$ be the Lebesgue measure on $G$. Let $\mathcal H = L^2(\R)$. Define the short-time Fourier transform with window $\pi: \R^2 \to L^2(\R)$ by 
\[
\pi(t,s)(\psi) = e^{2\pi i t \cdot} \psi(\cdot-s), \qquad t,s\in \R, \psi\in L^2(\R).
\]
Then for any $f, \psi \in L^2(\R)$, we have the well-known identity
\begin{equation}\label{gi}
\int_{\R}\int_{\R} | \lan f,\pi(s,t)(\psi)\ran|^2 dsdt = ||f||^2 ||\psi||^2.
\end{equation}
Technically, $\pi$ is not a unitary representation. However, translation in time and frequency commute up to a multiplicative factor, so $\pi$ is a projective unitary representation. By \eqref{gi}, if $||\psi||=1$, then $\{\pi(t,s)(\psi)\}_{(t,s) \in \R^2}$ is a continuous Parseval frame in $L^2(\R)$. Invoking Theorem \ref{db} shows the existence of a sampling sequence $\{(t_n,s_n)\}_{n\in \N}$ such that (non-uniformly spaced) Gabor system $\{\pi(t_n,s_n)(\psi)\}_{n\in\N}$ is a frame. With trivial modifications, the above example also holds for higher dimensional Gabor frames in $L^2(\R^d)$.
\end{example}

Due to the results of Feichtinger and Janssen \cite{FJ}, it is not true that any sufficiently fine lattice produces a Gabor frame for a general Gabor window $\psi\in L^2(\R)$. However, by the results of Feichtinger and Gr\"ochenig \cite{FG1, FG2}, a sufficiently well-behaved window $\psi$ in Feichtinger's algebra $M^{1,1}(\R)$ induces a Gabor frame for all sufficiently fine choices of time-frequency lattices.

\begin{example}[Wavelet frames] 
Let $G$ be the affine $ax+b$ group, which is a semidirect product of the translation group $\R$ by the full dilation group $\R_*$. Then, $d\mu(a,b)=|a|^{-2}dadb$ is the left Haar measure on $G$. Define the continuous wavelet transform $\pi: \R^2 \to L^2(\R)$ by
\[
\pi(a,b)(\psi) = |a|^{-1/2} \psi\bigg( \frac{\cdot - b}a \bigg), \qquad a\ne 0, b \in \R, \psi \in L^2(\R).
\]
Then, $\pi$ is a square integrable representation of $G$.
If $\psi$ satisfies the admissibility condition
\[
\int_{\R \setminus \{0\} } \frac{|\hat \psi(\xi)|^2}{|\xi|} d\xi =1,
\]
then $\{\pi(a,b)(\psi)\}_{(a,b) \in G}$ is a continuous Parseval frame in $L^2(\R)$ known as a continuous wavelet.
Again, invoking Theorem \ref{db} shows the existence of a sampling sequence $\{(a_n,b_n)\}_{n\in \N}$ corresponding to a discrete (albeit non-uniformly spaced) wavelet frame $\{\pi(a_n,b_n)(\psi)\}_{n\in\N}$.
\end{example}

\begin{acknowledgement}
The author was partially supported by the NSF grant DMS-1665056 and by a grant from the Simons Foundation \#426295. The author is grateful for useful comments of the referees and for an inspiring conversation with Hans Feichtinger.
\end{acknowledgement}

\bibliographystyle{amsplain}

\end{document}